\documentclass[11pt]{article}
\usepackage{graphicx}
\usepackage{amsmath, amssymb, amsthm, mathtools}
\usepackage{authblk}
\usepackage{hyperref, enumerate,enumitem}
\usepackage{nicematrix}
\newtheorem{theorem}{Theorem}[section]
\newtheorem{lemma}[theorem]{Lemma}

\newtheorem{proposition}[theorem]{Proposition}
\newtheorem{corollary}[theorem]{Corollary}
\theoremstyle{definition}
\newtheorem{definition}[theorem]{Definition}
\newtheorem{remark}[theorem]{Remark}

\newtheorem{example}[theorem]{Example}
\usepackage{mathrsfs}
\usepackage{tikz}
\usetikzlibrary{intersections}
\newcommand{\C}{\mathscr C}
\newcommand{\mcM}{\mathcal M}
\newcommand{\F}{\mathbb F}
\newcommand{\h}{\textnormal{H}}
\title{\textbf{Intersecting Codes and the Connectivity of $q$-Matroids}}
\newcommand{\mA}{\mathcal A}
\DeclareMathOperator{\rk}{rk}
\newcommand{\mcL}{\mathcal L}
\usepackage{palatino}
\usepackage[margin=3cm]{geometry}
\author{Fabrizio Conca, Benjamin Jany, Alberto Ravagnani}
\affil{Eindhoven University of Technology}
\usepackage{cite}

\begin{document}

\maketitle

\begin{abstract}
We investigate the structure of intersecting error-correcting codes, with a 
particular focus on their connection to matroid theory. We establish properties 
and bounds for intersecting codes with the Hamming metric and illustrate how these distinguish the 
subfamily of minimal codes within the family of intersecting codes. We prove 
that the property of a code being intersecting is characterized by the 
matroid-theoretic notion of vertical connectivity, showing that intersecting 
codes are precisely those achieving the highest possible value of this parameter.
We then introduce the concept of vertical connectivity for $q$-matroids
and link it to the theory of  intersecting 
codes endowed with the rank metric.
\end{abstract}

\medskip

\section{Introduction}

Intersecting codes are linear codes in which the Hamming supports of any two non-zero codewords have a non-empty intersection. They are a classical object in coding theory, first introduced in \cite{KatonaSrivastava1983}, with the binary case having been thoroughly investigated. Indeed, over the binary field the class of intersecting codes coincides with the class of minimal codes, which are closely related to secret sharing schemes and frameproof codes \cite{blackburn}. Intersecting codes beyond the binary case, and their links to finite geometry, have garnered interest more recently; see, for example, \cite{borello2024geometryintersectingcodesapplications}.

The contributions of this paper are threefold. In the first part, we establish structural properties and bounds for intersecting codes using methods ranging from coding theory and density arguments to extremal combinatorics. We also compare intersecting codes with the proper subclass of minimal codes, highlighting the differences in their typical behavior with respect to the properties under consideration.

In the second part, we provide a matroid-theoretic characterization of the property of being an intersecting code. We show that the relevant parameter in this context is the vertical connectivity of the matroid associated with the code, and that intersecting codes are precisely those that attain the maximum possible value for this parameter.

The third part of the paper is devoted to identifying and establishing the $q$-analogues of these results. We propose a definition of vertical connectivity for $q$-matroids, which captures the property of a rank-metric code being intersecting in the sense of~\cite{borello2024geometryintersectingcodesapplications}. A main result in this context is that a rank-metric code is intersecting if and only if the vertical connectivity of the associated $q$-matroid equals its dimension (the largest possible value).

The remainder of this paper is organized as follows. In Section~\ref{sec:2} we introduce intersecting codes and minimal codes, and survey their main properties.
In Section~\ref{sec:3} we establish properties and bounds for intersecting codes and compare them with minimal codes. Section~\ref{sec:4} is devoted to the link between matroid connectivity and intersecting codes. Finally, in Section~\ref{sec:5} we introduce the notion of vertical connectivity for $q$-matroids and link it with the class of intersecting rank-metric codes.

\section{Intersecting codes} \label{sec:2}

In this section, we establish the terminology and the notation for the rest of the paper. We also recall some known results on intersecting codes that we will need repeatedly. 

Throughout the paper, $\F_q$ denotes the finite field of order $q$.
The 
\textbf{Hamming support} of $x \in \F_q^n$ is $\sigma^{\h}(x) := \{i \in [n] \, : \, x_i \neq 0\}$ and its \textbf{Hamming weight} is $w^\text H (x) := |\sigma(x)|$.
The \textbf{Hamming distance}
between $x, y \in \F_q^n$ is defined by $d^{\h}(x,y) := w^{\h}(x-y)$. A \textbf{code} is a linear subspace $\C \le \F_q^n$, endowed with the Hamming metric, and its \textbf{minimum distance} is $d :=\min \{w^{\h}(x) \mid x \in \C \setminus \{0\}\}$, where the zero code $\{0\}$ has minimum distance $n+1$. The elements of $\C$ are called \textbf{codewords}.
A  $[n,k,d]_q$ code is a code $\C \le  \F_q^n$, with dimension $k$ and minimum distance $d$. Finally the \textbf{Hamming support} of a code $\C$ is $\sigma^{\h}(\C) := \bigcup_{x \in \C} \sigma^{\h}(x)$.

In practice, it is often desirable to construct codes of high dimension and with a large minimum distance. However, there is a tradeoff between the two parameters, captured for instance by the celebrated Singleton bound:

\begin{equation}\label{eq:Singletonbd}
    d \leq n-k+1
\end{equation}
for any $[n,k,d]_q$ code.
Codes for which the Singleton bound is tight are called \textit{Maximum Distance Separable} (MDS) codes.

In this article we are interested in intersecting codes, which are codes whose codewords' supports have an additional combinatorial property. 

\begin{definition}
    A code $\C \leq \F_q^n$ is \textbf{$t$-intersecting} if for all $v, w \in \C \setminus\{0\}$ we have 
    $$|\sigma^{\h}(v) \cap \sigma^{\h}(w)| \geq t.$$
    We refer to a $1$-intersecting code as an \textbf{intersecting} code.
\end{definition}
Note that if a code $\C$ is $t$-intersecting, then it is also $\ell$-intersecting for $1\leq \ell \leq t$. 
We now show that for some parameter choices, MDS codes form a class of intersecting codes. 

We remark that in the literature there exists a different notion of intersecting code~\cite{COHNEN200375}, where a code is said to be \textbf{$t$-wise intersecting} if any $t$ linearly independent codewords have intersecting supports. What we call an intersecting code is  a $2$-wise intersecting code according to~\cite{COHNEN200375}.

 \begin{example}\label{MDSint}
 Let $\C \leq \F_q^n$ be an $[n,k,d]_q$ MDS code. 
    It is well known that if $\C$ is MDS of minimum distance $d$, then for all $S \subseteq [n]$ such that $|S| = d$, there exist $x \in \C$ such that $\sigma^{\h}(x) = S$. 
    Therefore, one can conclude that $\C$ is intersecting if and only if $d > n/2$, or equivalently  $n\geq 2k-1$ (using the fact that $d = n-k+1$). More precisely, such a code is $(n-2d)$-intersecting, as this is the minimum size of the intersection of two supports of size~$d$.
\end{example}

Another interesting example of intersecting codes are \textit{minimal codes}, which have been intensively studied in their own right; see~\cite{ashbarg,HENG2018176, alfarano2020combinatorialperspectivesminimalcodes, alfarano2019geometric,bartoli2019inductive,
bartoli2020cutting,
bonini2020minimal,
MR3163591,
kiermaier2019minimum,
lu2019parameters,
Massey,
tang2019full} among many others.  

\begin{definition}
  Let $\C \le \F_q^n$. An element $v \in \C$ is a \textbf{minimal codeword} if for all $w \in \C$ such that $\sigma^{\h}(w) \subseteq \sigma^{\h} (v)$ we have $w = \lambda v$ for some $\lambda \in \F_q$. The code $\C$ is a \textbf{minimal code} if all of its codewords are minimal. 
\end{definition}

Note that if a code $\C$ is minimal, then for any linearly independent $x, y \in \C$ we have $\sigma^{\h}(x) \cap \sigma^{\h}(y) \neq \emptyset$, otherwise we would have $\sigma^{\h}(x) \subsetneq \sigma^{\h}(x+y)$, contradicting minimality. Hence minimal codes are intersecting. It is well known that intersecting codes and minimal codes are equivalent if and only if $q=2$. We will discuss in further details the differences between these two classes of codes in Section \ref{sec:3}.
 
We conclude this section by recalling some basic results about intersecting codes and known bounds on their parameters. 
As a first observation, codes with large enough minimum distance are intersecting. 
\begin{lemma}
    Let $\C$ be an $[n, k, d]_q$ code. If $2d > n$, then $\C$ is intersecting.
\end{lemma}

 In \cite{borello2024geometryintersectingcodesapplications}, the authors define  $i(k,q)$ to be the size of the smallest non-$2$-cohyperplanar set in $PG(k -1, q)$, which in coding terms is the the length of the shortest linear intersecting code  of dimension $k$ over $\F_q$. They then prove a Plotkin-like bound for intersecting codes.

\begin{theorem}[see~\text{\cite[Theorem~3.4]{borello2024geometryintersectingcodesapplications}}]
    For $1\leq t \leq k$,
    \[i(k,q) \geq k+\frac{q^t-1}{q^t-q^{t-1}}(k-t).\]
    Asymptotically, one has 
    \[ \liminf_{k\to\infty} \frac{i(k,q)}{k} \geq 2 + \frac 1 {q-1}.\]
\end{theorem}

\section{Intersecting versus minimal codes}\label{sec:3}

Since minimal codes are a subclass of intersecting codes (and coincide when $q=2$), it is natural to ask how different these two classes are.
In this section, we compare in detail some of the differences between intersecting codes and minimal codes. More precisely, we show that minimal codes require a significantly more rigid combinatorial structure than the larger class of intersecting codes. 
We prove various results on the structure of intersecting codes, including a density result and bounds on their parameters. 

We start by establishing bounds on the parameters of intersecting codes. 
 The first one we show can be found as Theorem 2.17 in \cite{borello2024geometryintersectingcodesapplications} using finite geometry, but here we present it purely in coding theory terms. 
\begin{proposition}\label{dgeqk}
    Let $\C \le \F_q^n$ be a linear intersecting code with minimum distance $d$ and dimension $k$.
    Then $d\geq k$. In particular, $n\geq 2k-1$.
\end{proposition}
\begin{proof}
    Let $x\in \C$ be a codeword of minimum Hamming weight $d$ and consider the projection $\pi_{\sigma^{\h}(x)}: \C \to \pi_{\sigma^{\h}(x)}(\C)$ onto its support. Observe that, since $\C$ is intersecting, its kernel $\ker(\pi_{\sigma^{\h}(x)})=\{y \in \C \mid \sigma^{\h}(y)\subseteq [n] \setminus \sigma(x)\}$ is trivial. Therefore,
    \[k=\dim(\pi_{\sigma(x)}(\C)) \leq |\sigma(x)|=d.\]
    Combining this with the Singleton bound yields
    $k\leq d \leq n-k+1$, hence $n\geq 2k-1$.
\end{proof} 

As an immediate corollary of Proposition \ref{dgeqk} we get the following bound for intersecting codes involving the field size. 

\begin{corollary}\label{cor:griesmertypebound}
     Let $\C \le \F_q^n$ be an intersecting code with minimum distance $d$ and dimension~$k$. Then 
     \[n \geq \sum_{j=0}^{k-1} \left \lceil \frac k {q^j} \right \rceil .\]
\end{corollary}
\begin{proof}
  It is sufficient to combine the Griesmer bound
  $n \geq \sum_{j=0}^{k-1} \left \lceil \frac{d}{q^j} \right \rceil$ 
  with Proposition~\ref{dgeqk}.
\end{proof}

For minimal codes, the bound of Proposition \ref{dgeqk} can be made sharper taking into account the field size. 

\begin{theorem}[see~\text{\cite[Theorem~4.3]{alfarano2019geometric}}]\label{thm:bdmincode}
    Let $\C \leq \F_q^n$ be a minimal code with minimal distance $d$ and dimension $k$. Then 
  \[d \geq k+q-2. \]
\end{theorem}
The examples of intersecting codes we have shown for now were all coming from the class of MDS codes. However, not all intersecting codes are MDS.
\begin{example}
    Let $\C$ be the ternary code generated by 
    \[G =\begin{pmatrix}
        1 & 0 & 0 & 1 & 0 & 2 \\
        0 & 1 & 0 & 2 & 2 & 1 \\
        0 & 0 & 1 & 1 & 1 & 1
    \end{pmatrix}.\]
    It can be shown that $\C$ is an intersecting code with parameters $[6,3,3]$. In particular $\C$ meets the bound of proposition \ref{dgeqk} with equality. This code is an example of an intersecting code that is neither minimal nor MDS.
\end{example}

Natural questions are to determine if there exist intersecting codes for which the bounds in Proposition \ref{dgeqk} and Corollary \ref{cor:griesmertypebound} are tight, as well if there exist intersecting codes for which the bound of Theorem~\ref{thm:bdmincode} does not hold. The next example answers both questions.
We furthermore observe that in the binary case,  Theorem \ref{thm:bdmincode} and  Proposition~\ref{dgeqk} coincide, reflecting the fact that intersecting binary codes are precisely the minimal codes. 

\begin{example}\label{nkk}
    If $n\leq 2k-2$, the Singleton bound precludes the existence of an $[n, k, k]$ code, regardless of field size.
    Then suppose  that $n\geq 2k-1$.
    For any $q$ large enough there exists a $[2k-1, k, k]_q$ MDS code $\C$, for example a Reed-Solomon code, which by Example \ref{MDSint}, is intersecting. Those codes achieve both the bound $d = k$ as well as $ n = 2k-1$.  Furthermore if $d = k$, the bound of \ref{cor:griesmertypebound} is precisely the Griesmer bound and it is well known that MDS codes achieve the Griesmer bound with equality. Thus the bound \ref{cor:griesmertypebound} is also tight.
    Finally, because those MDS codes exist over any field of large enough size, this example also shows that intersecting code do not need to satisfy the bound of Theorem \ref{thm:bdmincode}. 
\end{example}

An interesting consequence of Theorem \ref{thm:bdmincode} is that for $q >  d-k+2$, minimal codes do not exist. Since minimal codes are intersecting codes, one may wonder what is the density of intersecting code of large alphabet size. We answer this question with the following result.

\begin{theorem}\label{thm:dens}
Suppose $k>1$ and $n\geq 2k-1$. Then intersecting codes are dense within $[n,k]_q$ codes as $q\to\infty$.
On the other hand, if $n < 2k-1$, then  intersecting codes are sparse within $[n,k]_q$ codes as $q\to\infty$.
\end{theorem}
\begin{proof}
 
    If $n\geq 2k-1$, MDS codes are intersecting. By \cite[Corollary 5.3]{RavagnaniByrne2018}, MDS codes are dense as $q \to \infty$, and in particular intersecting codes are dense.

    If $n < 2k-1$, the classes of intersecting codes and that of MDS codes are disjoint. Since MDS codes are dense as $q\to\infty$, again by~\cite[Corollary 5.3]{RavagnaniByrne2018}, intersecting codes are sparse in this parameter range.
\end{proof}

\begin{remark}
    It is shown in \cite{MDSweights} that, aside from some exceptional cases, an MDS code of parameters $[n,k,d]$ contains codewords of all weights $d, d+1, \ldots, n$. Since every set of size $d$ is the support of some codeword, MDS codes are non-minimal. By the same reasoning of~\ref{thm:dens}, $[n,k]$ minimal codes are sparse among $[n,k]$ codes as $q\to\infty$. 
\end{remark}

 Intersecting codes can easily be related to the well studied combinatorial notion of intersecting families. We recall their definition.

\begin{definition}
    Let $X$ be a set and $\mathcal A$ be a collection of subsets of $X$. We say $\mathcal A$ is an \emph{intersecting family} if
    \[A \cap B \neq \emptyset \quad \text{for all } A, B \in \mathcal A.\]
    Given an integer $t \geq 1$, we say $\mA$ is $t$-\emph{intersecting} if
    \[|A \cap B| \geq t  \quad \text{for all } A, B \in \mathcal A.\]
\end{definition}
 Using the above notion one can equivalently define a $t$-intersecting code as a code whose collection of non-empty supports  forms a $t$-intersecting family.  Motivated by this, we recall the famous  Erd\H{o}s-Ko-Rado Theorem.
\begin{theorem}[Erd\H{o}s-Ko-Rado Theorem; see \cite{erdoskorado}]
    Let $\mathcal {A}$ be a $t$-intersecting family of $r$-element subsets of $[n]$, with $n \geq (t+1)(r-t+1)$. Then 
    \[|\mathcal A| \leq \binom{n-t}{r-t}.\]
\end{theorem}
An instance of family $\mathcal A$ as in the statement of the   Erd\H{o}s-Ko-Rado Theorem is the collection of supports of a given size in an intersecting code. 
For $0\leq i\leq n$ and for a code $\C \le \F_q^n$, we
let
\[W_i(\C)=|\{x\in \C \mid w^{H}(x)=i\}|.\]

\begin{proposition}\label{lemma:upWd}
    Let $\C \le \F_q^n$ be a linear $t$-intersecting code with minimum distance $d$ and $n \geq(t+1)(d-t+1) $. Then,
    \[W_d(\C)\leq (q-1) \binom {n-t}{ d-t}   .\]
\end{proposition}
\begin{proof}
    By definition, the supports of a $t$-intersecting code form an intersecting family. 
    Observe that two codewords $x, y \in \C$ of weight $d$ having the same support are necessarily linearly dependent. To see this, take $i\in \sigma(x)=\sigma(y)$ and consider $z=y_i x -x_iy\in \C$. Clearly $\sigma(z)\subseteq \sigma(x)=\sigma(y)$, hence $w(z)\leq d$. However, $z_i=0$, meaning $w(z)<d$, hence $z=0$ and $x, y$ are linearly dependent.
    The number of distinct $d$-supports in $\C$ is therefore $W_d(\C)/(q-1)$. 
    By the  Erd\H{o}s-Ko-Rado Theorem we therefore have
    \[ \frac{W_d(\C)}{q-1}\leq \binom{n-t}{d-t},\]
    and the result follows.
\end{proof}

If instead of intersecting codes we restrict ourselves to minimal codes, the previous result can be generalized to weights higher than the minimum distance. By definition, a minimal code cannot have linearly independent codewords sharing the same support. This simple observation gives a bound on $W_i(\C)$ for $d\leq i \leq n$ when $\C$ is a minimal code:
\begin{align} \label{eq:Witrivial}
    W_i(\C) \leq (q-1)\binom{n}{i}.
\end{align}
We can actually refine this bound  by applying the  Erd\H{o}s-Ko-Rado Theorem, using the same idea as in Proposition \ref{lemma:upWd}.

\begin{proposition}\label{prop:minimalweight}
     For a minimal, $t$-intersecting linear code $\C \leq \F_q^n$, with $d \leq t-1+\frac{n}{t+1}$, we have
    \[W_i(\C) \leq (q-1)\binom{n-t}{i-t} \quad \text{ for all } i \in \left\{d, \ldots, t-1 + \left\lfloor \frac{n}{t+1} \right\rfloor \right\}.\]
\end{proposition}
\begin{proof}
    Since $\C$ is minimal, for all $i \in \{d, \ldots, n\}$,  the number of distinct supports of size $i$ is $W_i(\C)/(q-1)$. The argument is the same as in the proof of Proposition \ref{lemma:upWd}. The  Erd\H{o}s-Ko-Rado theorem then gives
    \begin{equation*}
        W_i(\C) \leq (q-1)\binom{n-t}{i-t} 
    \end{equation*} 
    for weights $i\leq  t-1 + \lfloor \frac{n}{t+1} \rfloor. $
\end{proof}

 Recall that a family of sets $\mA$ is \emph{non-trivial} if $\bigcap_{A\in\mA}A=\emptyset$. For non-trivial families of $r$-element sets, the  Erd\H{o}s-Ko-Rado bound is improved by the Hilton-Milner Theorem.
\begin{theorem}[Hilton–Milner, \cite{hiltonmilner}]\label{thm:HiltonMilner}
     Let $\mathcal {A}$ be a non-trivial intersecting family of $r$-element subsets of $ [n]$, with $n \geq  2 r $. Then
     \[|\mathcal A| \leq \binom{n-1}{r-1}-\binom{n-r-1}{r-1}+1.\]
\end{theorem}

The Hilton-Milner can be used to improve on the bound of Proposition~\ref{lemma:upWd} when the supports for a  non-trivial family. 

\begin{corollary}
     Let $\C \le \F_q^n$ be a linear intersecting code with minimum distance $d$, with $n \geq 2d$. Suppose there is no coordinate that is non-zero on all codewords of weight $d$. Then
    \[W_d(\C)\leq (q-1)\left( \binom{n-1}{d-1}-\binom{n-d-1}{d-1}+1 \right) .\]
\end{corollary}
\begin{proof}
    Under these assumptions the family $\mA$ of weight $d$ codewords of $\C$ is a non-trivial intersecting family. Hence, by \ref{thm:HiltonMilner}, 
    \begin{equation*}
        \frac{W_d(\C)}{q-1}=|\mA|\leq \binom{n-1}{d-1}-\binom{n-d-1}{d-1}+1. \qedhere\end{equation*}
\end{proof}

Similarly to our previous results, if the code is minimal, we can apply Theorem \ref{thm:HiltonMilner} on the number of codewords of weight $i > d$.

\begin{corollary}
     Let $\C \le \F_q^n$ be a minimal code with minimum distance $d$, with $n \geq 2i$. Suppose there is no coordinate that is non-zero on all codewords of weight $i$ for a given $d\leq  i\leq n/2$. Then
    \[W_i(\C)\leq (q-1)\left( \binom{n-1}{i-1}-\binom{n-i-1}{i-1}+1 \right) .\]
\end{corollary}

\begin{proof}
    The statement follows from the fact that the collection of support of codewords of weight $d \leq i \leq n/2$ forms a non-trivial intersecting family. 
\end{proof}

\section{Matroid connectivity}  \label{sec:4}
Linear codes have been extensively investigated in connection with matroids; see~\cite{barg,britz2007higher,jurrius2013codes,britz2010code, ghorpade2022shellability, J2013} among many others. Some code properties are fully captured by the associated matroid, such as the property of being an MDS code. In this section, 
we link the notion intersecting codes to matroid theory, and more specifically matroid connectivity. We start by recalling the definition of a matroid (we give the definition via the rank function) and other useful definitions. For more information on matroid we refer the reader to \cite{Oxley}.

\begin{definition} 
   Let $E$ be a finite set and  $r: 2^{E} \rightarrow \mathbb N$ a function on its power set. The pair $M = (E , r)$ is a \emph{{matroid}} if, for all $A, B \in 2^E$, it satisfies:
    \begin{itemize}
        \item[(R1)]  $0 \leq r(A) \leq |A|$.
        \item[(R2)] if $A \subseteq B$ then $r(A) \leq r(B)$.
        \item[(R3)] $r(A \cap B) + r(A \cup B) \leq r(A) + r(B)$ (submodularity).
    \end{itemize}
    The set $E$ is called the \emph{{ground set}} of the matroid and $r$ is its \emph{{rank function}}.
\end{definition}

We will need some basic matroid-theoretic notions, which are summarized in the definition below.
\begin{definition}
Let $M=(E, r)$ be a matroid and $X\subseteq E$. Then
    \begin{itemize}
         \item  $X\subseteq E$ is an \emph{independent set} for $M$ if $r(E)=|E|$, and \emph{dependent} if $r(E)<|E|$;
         \item  A subset $X\subseteq E$ is a \emph{circuit}  for $M$ if $r(E)=|E|-1$, and all its proper subsets are independent.
    \end{itemize}
    Moreover,
    \begin{itemize}
        \item The \emph{dual matroid} of $M$ is $M^*=(E, r^*)$, with rank function given by \[r^*(X)=|X|-r(E)+r(X^c).\]
            \item $X$ is a \emph{cocircuit} of  $M$ if it is a circuit for  $M^*$.
       
    \end{itemize}
\end{definition}

A common way to construct a matroid is by looking at the linear dependencies between columns of a matrix. Suppose $\mathbb F$ is a field and $G$ is a $k\times n$ matrix over $\mathbb F$. Indexing the columns of $G$ as $[n]$, we can build a matroid $M(G)=([n], r)$, such that $r:2^{[n]}\to\mathbb N$ and $r(A) = \rk (G_A)$, where $\rk$ is the classical rank of a matrix and $G_A$ is the submatrix of $G$ consisting of columns indexed by elements of $A$.  A matroid that can be constructed in this way starting from a matrix is called a \emph{representable} matroid. Using the notion of generator matrices for vector spaces, we can associate a matroid to a linear code.
\begin{definition}
    Let $\C \subseteq \F_q^n$ be a linear code. Its \textbf{associated matroid} is $M(G)$, where $G$ is a generator matrix for $\C$. This independent of the choice of $G$. 

\end{definition}

Note that one can easily check that the associated matroid to a code is independent of the choice of generator matrix, hence the above definition is well-defined. 

Connectivity is a classical graph-theoretic notion that measures how hard it is to disconnect a graph. A subset of the vertex set of connected a graph is called a \emph{vertex cut} if removing the corresponding vertices and deleting all edges attached to them produces a disconnected graph. There are several ways to extend this idea to matroids, most notably \emph{Tutte} and \emph{vertical} connectivity. In this paper we will focus on \emph{vertical} connectivity, which was first introduced independently in 1981 in three different papers \cite{Inukaivert, cunningham1981matroid, Oxley1981}. We will sometimes refer to it as simply \emph{connectivity}. For a thorough treatment of matroid connectivity and its various incarnations we refer the reader to \cite[Chapter~8]{Oxley}.

\begin{definition}
The \textbf{connectivity function} of a matroid $M=(E, r)$ is $\lambda_M: 2^E\to \mathbb N$ 
\[\lambda_M(X)\coloneqq r(X)+r(E\setminus X)-r(M)=r(X)+r^*(X)-|X|.\]
The submodularity of the rank function ensures that $\lambda_M(X)\geq 0$ for all $X\subseteq E$.
\end{definition}
In this context, the analog of a vertex-cut is called a vertical separation, and is defined as follows.

\begin{definition}
   Let $M$ be  a matroid of rank $k$. For a positive integer $t$, a partition $(X, X^c )$ of the ground set of $M$ is
a \textbf{vertical $t$-separation} of $M$ if $\lambda_M(X) < t$ and $\min \{r(X), r(Y)\} \geq t$. 
\end{definition}
We remark here that the word ``vertical'' is the adjective corresponding to
the noun ``vertex''. 
We are now ready to give the definition of vertical connectivity for a matroid. Recall that the minimal size of a vertex cut is the graph-theoretic notion of connectivity that we want to apply to matroids. Having defined vertical separations, the vertical connectivity of a matroid is defined as follows.

\begin{definition}
    Let $M=(E, r)$ be a matroid that has a vertical separation for some $1 \leq j\leq r(M)-1$. The \textbf{vertical connectivity} of $M$ is 
    \[\kappa(M)=\min\{j \mid M \text{ has a vertical $j$-separation}\}.\]
    If $M$ has no vertical separations we set \[\kappa(M)=r(E),\]
    in which case we say $M$ is \emph{fully connected}.
\end{definition}

\begin{remark}
    If $\Gamma$ is a connected graph and $M(\Gamma)$ is the corresponding circuit matroid (see for example \cite[Chapter~1]{Oxley}), then \[\kappa(M(\Gamma))=\kappa(\Gamma),\]
    which is the main reason why vertical connectivity is introduced by the authors in \cite{Inukaivert, cunningham1981matroid, Oxley1981}, in contrast to Tutte connectivity, which is not as well behaved with respect to graphic matroids. The main drawback of vertical connectivity is that, unlike Tutte connectivity, it is not invariant under duality, as can be seen by looking at uniform matroids.
\end{remark}

\begin{example} \label{vert_unif}
    Consider the uniform matroid $U_{k,n}$, its vertical connectivity is given by
    \[\kappa(U_{k,n})=
    \begin{cases}
        n-k+1 & \text{if } n\leq 2k-2, \\
        k & \text{otherwise}.
    \end{cases}\]
    In particular, $\kappa(U_{3, 4})=2$, while $\kappa(U_{3,4}^*)=\kappa(U_{1,4})=1$.
\end{example}

There is a relation between the existence of a vertical separation and the cocircuit structure of a matroid. Although the following result is well known (see for example \cite[Lemma~8.6.2]{Oxley}), we include a proof for self-containment.

\begin{proposition}\label{prop:disjcirc}
     Let $M = (E, r)$ be a matroid. Then $M$ has a vertical  $t$-separation for some $t \leq r(M)-1$ if and only if $M$ has two disjoint cocircuits. 
\end{proposition}
\begin{proof}
    If $M$ has a vertical $t$-separation $(X, X^c )$, then
    \[t+t-k\leq r(X)+r(X^c)-k\leq t-1,\] which implies $t \leq k-1$, $r(X)\leq k-1$ and $r(X^c)\leq k-1$.
    Therefore,
\begin{align*}
    r^*(X)&=|X|-r(M)+r(X^c)  \leq |X|-k+k-1  =|X|-1, \\ 
    r^*(X^c)&=|X^c|-r(M)+r(X)  \leq |X^c|-k+k-1 =|X^c|-1.
\end{align*} 
    Thus $X, X^c$ are both  dependent in $M^*$ and hence contain disjoint cocircuits of $M$.
    Conversely, suppose $M$ has two disjoint cocircuits $C_1^*$ and $C_2^*$. Then
    \[ r({C_1^*}^c)=r({C_2^*}^c)=r(M)-1. \]
     Let $t=r(C_1^*)$. Since $C_1^*$ and $C_2^*$ are disjoint,
    \[ t = r(C_1^*) \leq r({C_2^*}^c)=r(M)-1.\]
    Consider the partition $(C_1^*, {C_1^*}^c)$. We have that 
    \[ \min\{r(C_1^*), r({C_1^*}^c)\}=\min\{t, r(M)-1\}=t \]
    and 
    \[
    \lambda(C_1^*)=r(C_1^*) + r({C_1^*}^c) -r(M) = t+r(M)-1-r(M)=t-1<t,
    \]
    therefore $(C_1^*, {C_1^*}^c)$ is a vertical $t$-separation.
\end{proof}

The vertical connectivity of the matroid associated to a linear code gives information about the structure of its generator matrices. Given a linear code $\C$ and a set of coordinates $A\subseteq [n]$ we denote by $\C(A)=\{v\in \C \mid \sigma^{\h}(v)\subseteq A\}$. Clearly, $\C(A)$ is a (potentially $0$) subspace of $\C$.

\begin{proposition}
    Let $M$ be the matroid induced by an $[n, k, d]$-code $\C \leq \F_q^n$. If $M$ has vertical connectivity $t<\dim(\C)$ then, up to permutation of columns, there exists a generator matrix $G$ of $\C$ of the form 
    $$G =  
        \left(
    \begin{array}{c} 
      \begin{array}{c | c} 
        G_1 & 0\\ 
        \hline
         0 & G_2 \\
      \end{array}   \\ 
      \hline 
     -  B -
     \end{array} 
    \right),
    $$
    where $G_1, G_2$ combined have $k-(t-1)$ rows, and $B$ has $t-1$ rows.
\end{proposition}
\begin{proof}
Since $M$ has vertical connectivity $\kappa(M)=t<\dim(\C)$, by definition, there exists  a subset $A\subseteq [n]$ with $r(A)\geq t$,  $r(A^c)\geq t$  such that 
\[\lambda(A)=r(A)+r(A^c)-r(M)<t.\]
Therefore, $\lambda(A)<t\leq r(A).$ By the minimality in the definition of connectivity, $\lambda(A)\geq t-1.$ We conclude that $\lambda(A)=t-1.$

 Recall that $r(A)=r(M)-\dim(\C(A^c))=k-\dim(\C(A^c))$.
  Using this fact we obtain that 
\begin{align} \label{block_struct}
   0 \leq \dim \C(A) +\dim \C(A^c) = k-(t-1), \text{ hence } k\geq t-1,
\end{align}
and, clearly, $\C(A)$ and $\C(A^c)$ are in direct sum.
Let $\mathscr B \leq \C$ be  a complement to $\C(A), \C(A^c)$, i.e 
\[\C = \C(A) \oplus \C(A^c) \oplus \mathscr B, \quad \dim(\mathscr B)=t-1. \]
If $k>t-1$ then at least one among $\C(A)$ and $\C(A^c)$ is a nonzero subspace. Without loss of generality, up to permuting some coordinates, suppose that $A=\{1, \ldots, m\}$, then a  generator matrix of $\C$ can be written in the form

   \begin{align*} \label{block_struct}
       G= \left(
    \begin{array}{c} 
      \begin{array}{c | c} 
        G_1 & 0\\ 
        \hline
         0 & G_2 \\
      \end{array}   \\ 
      \hline 
     -  B -
     \end{array} 
    \right),
   \end{align*}  
where $B$ has $t-1$ rows and spans $\mathscr B$, while $G_1$, $G_2$ combined have $k-(t-1)$ rows and each have at least $t$ columns (more precisely, the submatrix corresponding to the columns of $G$ that contain $G_1$ has rank at least $t$, and likewise for $G_2$), 
$\left(\begin{array}{c | c} 
        G_1 & 0
\end{array}\right)$
generates $\C(A)$, and $\left(\begin{array}{c | c} 
        0 & G_2
    \end{array} \right)$
generates $\C(A^c)$.
        If both blocks are nonempty, then $d(\C)\leq \min\{|A|, |A^c|\}$. If one of the two blocks is empty, say $G_2$, then $G$ is of the form 
\[   G = \left(
    \begin{array}{c} 
      \begin{array}{c | c} 
        G_1 & 0\\ 
      \end{array}   \\ 
      \hline 
     - B -
     \end{array}
     \right)
\]
and $d(C)\leq |A|$.
\end{proof}
The previous proposition shows in particular that, if a code is the direct sum of two smaller codes (i.e. its generator matrix can be expressed in a block-diagonal form), then the connectivity of the code is $1$. This is in line with the analog statement for matroids. 

\begin{remark}
    Note that if $\C$ has generator matrix of the  \begin{align*} \label{block_struct}
       G= \left(
    \begin{array}{c} 
      \begin{array}{c | c} 
        G_1 & 0\\ 
        \hline
         0 & G_2 \\
      \end{array}   \\ 
      \hline 
     -  B -
     \end{array} 
    \right),
   \end{align*}
   it is not necessarily true that $\kappa(M(G)) = t$ where $t$ is one plus the number of rows of $B$.  For example,
    \begin{align*} 
       G= \left(
    \begin{array}{c} 
      \begin{array}{ccc | ccc} 
        1 & 1 & 1 & 0 & 0 \\ 
        \hline
          0 & 0 & 0 & 1 & 1  \\
      \end{array}   \\ 
      \hline 
      \begin{array}{cccccc} 
        1 & 1 & 0 & 0 & 1 \\
        0 & 1 & 1 & 0 & 1 
     \end{array} 
     \end{array}
    \right) \in \F_2^{5\times 6},
   \end{align*}
   actually generates a code of connectivity $2$, as elementary operations on the rows produce a generator matrix of the same code of the form
   \begin{align*} 
       G'= \left(
    \begin{array}{c} 
      \begin{array}{ccc | ccc} 
        1 & 1 & 1 & 0 & 0 \\ 
         1 & 0 & 1 & 0 & 0 \\
        \hline
          0 & 0 & 0 & 1 & 1  \\
      \end{array}   \\ 
      \hline 
      \begin{array}{cccccc} 
        0 & 1 & 1 & 0 & 1 
     \end{array} 
     \end{array}
    \right).
   \end{align*}
\end{remark}

For a linear code $\C$, we denote by $\kappa(\C)$ the vertical connectivity of the associated matroid. Recall the following, which we prove for completeness.
\begin{lemma}\label{mincodeword}
    A codeword $x\in\C$ is minimal if and only if its support is a cocircuit in the corresponding matroid.
\end{lemma}

We can finally show that vertical connectivity characterizes the intersecting property of a code. 

\begin{theorem}\label{thm:conn_int}
     A code $\C$ is intersecting if and only if its vertical connectivity equals its dimension, $\kappa(\C)=\dim(\C)$.
\end{theorem}
\begin{proof}
    By definition, the connectivity of the code equals its dimension precisely when no $j$-separations exist for any $j\in \{1, \ldots, \dim(\C)-1\}$. By  Proposition~\ref{prop:disjcirc} this is equivalent to the corresponding matroid having no disjoint cocircuits. By Lemma~\ref{mincodeword}, this corresponds to $\C$ having no minimal codewords with disjoint supports. Naturally, this is equivalent to $\C$ not having any non-zero codewords with disjoint supports, i.e. $\C$ is intersecting.
\end{proof}

\section{Connectivity of $q$-matroids and  rank-metric codes}  \label{sec:5}

A natural  question is whether the link between matroid connectivity and intersecting codes
 extends to their respective $q$-analogs, i.e., \textit{$q$-matroids} and rank-metric codes. These two classes of objects have been intensively studied in connection with each other over the past decade, see for example~\cite{ghorpade2022shellability,gluesing2025cloud,jany2023projectivization,Jurriusqmatroid,degen2024most,alfarano2023critical,byrne2023cyclic,byrne2022constructions,alfarano2024cyclic}.

In this section we will recall the definition of a $q$-matroid and propose a notion of connectivity for $q$-matroids and a corresponding definition of intersecting rank-metric code.

 The rank weight $\operatorname{rk}(v)$ of a vector $v \in  \F_{q^m}^n$ is the $\F_q$-dimension of the
    $\F_q$-linear space generated by its entries.
    A \textbf{vector-rank-metric code} is an $\F_{q^m}$ -linear 
    subspace $\C\subseteq \F_{q^m}^n$ equipped with the rank distance
 \[d_{\rk}(u, v) = \operatorname{rk}(u-v).\]

For a detailed treatment of rank-metric codes we refer to \cite{Gorla2018}. In this section, we only consider codes under the rank-metric and no longer under the Hamming distance. For that reason, we omit the subscript under the rank distance function. We next recall the notion of rank support of a codeword.
\begin{definition}[Rank-metric support]
    Let $\{\gamma_1, \ldots, \gamma_m\}$ be a basis of $\F_{q^m}$ over $\F_q$, and fix $v=(v_1, \ldots, v_n)\in \F_{q^ m}^ n$. Let $\Gamma(v) = (a_{ij})\in \F_q^{m \times n}$ be the matrix such that, for all ,
    \[v_i=\sum_{j=1}^ m a_{ij}\gamma_j.\] The \textbf{rank-metric support} of $v\in \F_{q^m}^n$ is $\sigma^{\operatorname{rk}}(v)=\operatorname{colsp}(\Gamma(v)) \le \F_q^n$.
\end{definition}
This notion of support yields the notion of a vector rank-metric intersecting code by requiring that supports intersect non-trivially.

\begin{definition}
    Let $\C\subseteq \F_{q^m}^n$ be a vector-rank-metric code. We say $\C$ is \textbf{intersecting} if, for all nonzero $v,w \in \C$, $\sigma^{\operatorname{rk}}(v)\cap\sigma^{\operatorname{rk}}(w)\neq \{0\}$.
\end{definition}

Intersecting rank-metric codes were introduced and investigated in \cite{bartoli2025linearrankmetricintersectingcodes}, however, in their definition the authors consider the rowspace of matrix $\Gamma(v)$ as the support of $v$. 

Just like for classical matroids and intersecting codes, we will show in the remainder of this section how $q$-matroid vertical connectivity relates to  intersecting rank-metric codes.

The notion of $q$-matroid was reintroduced by Jurrius and Pellikaan in \cite{Jurriusqmatroid}. One of their motivations was  the fact that a rank-metric code can be associated to a $q$-matroid. We first recall the definition of q-matroids.

\begin{definition}[$q$-matroid]
     Let $E$ be a finite-dimensional vector space 
     and $\rho: \mathcal L(E) \rightarrow \mathbb N$, where $\mathcal{L}(E)$ is the collection of subspaces of $E$. The pair $\mcM = (E , \rho)$ is a \textbf{{$q$-matroid}} if for all $V, W \in \mathcal L(E)$ the following hold:
    \begin{itemize}
        \item[($q$R1)] $\rho(V) \leq \dim V$,
        \item[($q$R2)] if $V \leq W$ then $\rho(V) \leq \rho(W)$,
        \item[($q$R3)]$\rho(V \cap W) + \rho(V + W) \leq \rho(V) + \rho(W).$
    \end{itemize}
    We refer to $E$ as the \textbf{{ground space}} and $\rho$ as the \textbf{{rank function}} of the $q$-matroid.
    
    Two q-matroids $\mcM_1 = (E_1, \rho_1), \mcM_2 = (E_2, \rho_2)$ are said to be \textbf{equivalent}, and denoted $\mcM_1 \sim ~\mcM_2$, if there exists a lattice isomorphism $\varphi : \mcL(E_1) \rightarrow \mcL(E_2)$ such that,  $\rho_1(V)= \rho_2(\varphi(V))$  for all $V~\in~ \mcL(E_1)$.
\end{definition}

\begin{example}
    Let $E=\F_q^n$ and $k\leq n$. The uniform $q$-matroid is $U_{k,n}=(E, \rho)$, where 
    \[\rho(V)=\min\{\dim(V), k\}\]
    for all $V\le E$. 
\end{example}

Similarly to classical matroid, we can define the following. Given a $q$-matroid $\mcM = (E,\rho)$, a subspace $V \leq E$ is independent if $\rho(V) = \dim(V)$ otherwise it is dependent. A circuit is minimal dependent space w.r.t. inclusion. Just like for matroids, each of those collections of spaces fully determine the $q$-matroid, see \cite{Jurriusqmatroid,byrne2022constructions} for more detail.

At the end of their seminal paper where $q$-matroids are introduced, Jurris and Pellikaan propose a wish-list of results or notions that are known for classical matroids and could be extended to $q$-matroids.
Among them is the notion of connectivity for a $q$-matroid. We propose the following definition of vertical connectivity for $q$-matroids.

\begin{definition}[Vertical connectivity for $q$-matroids]\label{def:qconn}
    Let $M=(E, \rho)$ be a $q$-matroid and $t>0$ be an integer. A pair of subspaces $(A,V) \in \mathcal{L}(E) \times \mathcal{L}(E)$ is a \textbf{vertical $t$-separation} if :
    \begin{itemize}
        \item $A\cap V=\{0\}$, $A+V=E$ (i.e., $V$ is a complement of $A$),
        \item $\min\{\rho(A), \rho(V)\} \geq t$,
        \item  $\lambda(A, V)\coloneqq \rho(A)+\rho(V)-\rho(E)<t$.
    \end{itemize}
   We say $\mcM$ is \textbf{vertically $t$-connected} if it has no $j$-separations for $j\in \{1, \ldots, t-1\}$ and we define the \textbf{vertical connectivity} of $M$ as
    \[\kappa(\mathcal M)=\min\{t : \mathcal M \text{ has a vertical $t$-separation}\}.\]
    If $\mathcal M$ does not have any vertical separations, we set $\kappa(\mathcal M)=\rho(E)$.
\end{definition}

Before going further, we first show that our notion of vertical separation is $q$-matroid invariant, that is, it is preserved under $q$-matroid equivalence. 

\begin{proposition}\label{prop:conequiv}
    Given two equivalent $q$-matroids $\mcM_1 = (E_1, \rho_1)$ and $\mcM_2 = (E_2, \rho_2)$ and $\varphi : \mcL(E_1) \rightarrow\mcL(E_2)$ the rank-preserving lattice isomorphism.  A pair $(A, V) \in \mcL(E_1) \times \mcL(E_1)$ is a vertical $t$-separation in $\mcM_1$ if and only if $(\varphi(A), \varphi(V))$ is a vertical $t$-separation in $\mcM_2$.
    In particular  $\kappa(\mcM_1) = \kappa(\mcM_2)$. 
\end{proposition}

\begin{proof}
Let $(A, V)$ be a vertical $t$-separation of $\mcM_1$ and consider $(\varphi(A), \varphi(V))$. Since $\varphi$ is a lattice isomorphism, we have $\varphi(A) \cap \varphi(V) = \{0\}$ and $\varphi(A) + \varphi(V) = E_2.$ Furthermore, since $\varphi$ is rank-preserving, we have $\min\{\rho_2(\varphi(A)), \rho_2(\varphi(V))\} = \min\{\rho_1(A), \rho_1(V)\} \geq t$, and similarly $\lambda(\varphi(A),\varphi(V)) = \lambda(A, V) < t$. 
The converse follows in a similar way using the fact that $\varpi$ has an inverse. The claim $\kappa(\mcM_1) = \kappa(\mcM_2)$ then follows. 
\end{proof}

By applying our definition to uniform $q$-matroids, we obtain a direct analog of Example~\ref{vert_unif}.

\begin{example}[Vertical connectivity of uniform $q$-matroids]
    The vertical connectivity of the uniform $q$-matroid $U_{k, n}$ is given by
    \[
    \kappa(U_{k, n})=
    \begin{cases}   
        n-k+1 & \text{if } n \leq 2k-2, \\
        k & \text{otherwise}.
    \end{cases}
\]
To see this, first suppose  $n \leq 2k-2$.
Take $A \le \mathbb F_q^n$ of dimension $n-k+1$ and a complement~$V$ of $A$. We have that:
\begin{itemize}
    \item $\dim A = n-k+1 \leq k-1$, hence $r(A) = \dim A =n-k+1,$ 
    \item $\dim V = k-1$, hence $r(V)= \dim V = k-1,$ 
    \item $\min\{\rho(A), \rho(V) \} = n-k+1$,
    \item $\lambda(A, V)=\rho(A)+\rho(V)-\rho(E)=n-k+1+k-1-k=n-k<n-k+1$.
\end{itemize}
Therefore, $(A,V)$ is a separation and $\kappa(U_{k, n}) \leq n-k+1$. To see that equality holds, take $A \oplus V = E$ such that $\dim A = t \leq n-k$. Then $\lambda(A, V) = t -k+k =t$, i.e. $(A, V)$ is not a separation. Hence  $\kappa(U_{k, n}) = n-k+1$.
\end{example}

Just like in the classical case, the existence of a vertical separation is linked to the cocircuit structure of the $q$-matroid. One needs to be slightly more careful, since the duality now depends on a choice of inner product, but the spirit of the result is the same. 

\begin{definition}
    Let $\mcM=(E, \rho)$ be a $q$-matroid and fix an inner product $\perp$ on $E$ (i.e. a nondegenerate bilinear form on $E$). The dual $q$-matroid of $\mcM$ is $\mcM^* =(E, \rho^*)$, where 
    \[\rho^*(X)= \dim X -\rho(E)+\rho(X^\perp),\]
    and the orthogonal is taken with respect to the fixed inner product.
\end{definition}

It was shown in \cite[Thm. 42]{Jurriusqmatroid}  that if $\mcM^*$ and ${\mcM^{*}}'$ are dual $q$-matroids of $\mcM$ with respect to inner products $\perp, \perp'$ on $E$ respectively, then $\mcM^* \sim {\mcM^*}'$. For that reason, for the remaining of the paper and except if stated otherwise, $\mcM^*$ is the dual $q$-matroid w.r.t the standard inner product. 

Much like in the classical case, the existence of separations in a $q$-matroid is linked to the support structure of its dual, which we show next by establishing a $q$-analogue of Proposition~\ref{prop:disjcirc}. 

\begin{theorem}\label{prop:qmatconcirc}
    Let $\mcM = (E, \rho)$ be a $q$-matroid and $\mcM^*$ its dual $q$-matroid.  There exists a vertical $t$-separation, with $t \leq \rho(E)-1$ for $\mcM$ if and only if there exist two disjoint circuits in $\mcM^*$.

\end{theorem}

\begin{proof}
    First recall that given two inner products $\perp$, $\perp'$, the resulting and respective dual $q$-matroid $\mcM^*, \mcM^{*'}$ are equivalent.
    Since a rank-preserving lattice isomorphism maps circuits to circuits and preserves preserves intersections of two spaces, 
 to prove $\Rightarrow$, it is sufficient to show that the statement holds for one choice of inner product.

 \smallskip

\noindent($\Rightarrow$) Let $(A, V)$ be a vertical $t$-separation of $\mcM$. Then 
    \[t+t-r(E)\leq \rho(A)+\rho(V)-\rho(E)\leq t-1.\]
    Therefore 
    \begin{align*}
        t \leq \rho(E)-1
    \end{align*}
    and 
    \begin{align*}
        \rho(A)+\rho(V) \leq \rho(E)+t-1 \leq 2\rho(E)-2, \\
        \max\{\rho(A), \rho(V)\}\leq \rho(E)-1.
    \end{align*}
    Now fix an inner product with respect to which $V=A^\perp$ and consider the dual $q$-matroid $\mathcal M^*=(E, \rho^*)$, where
    \[\rho^*(X)= \dim X - \rho(E)+ \rho(X^\perp).\]
   Then
    \begin{align*}
        \rho^*(A)&= \dim A- \rho(E)+ \rho(A^\perp)   \leq \dim A -\rho(E)+\rho(E)-1  =\dim A-1, \\ 
        \rho^*(V)&= \dim V-\rho(M)+\rho(V^\perp)   \leq \dim V -\rho(E)+\rho(E)-1 =\dim V-1,
    \end{align*} 
    thus $A, V$ are both dependent in $\mathcal M^*$ and hence contain trivially intersecting circuits of $\mcM^*$.

\medskip

\noindent($\Leftarrow$) Let $B, W \leq E$ be disjoint circuits of $\mcM^*$. Now let $A, V \leq E$ such that $B \leq A$, $W \leq V$ and $A \oplus V = E$. Select an inner product $\perp'$ such that $A^{\perp'}= V$.  Consider $\widehat{\mcM} := (\mcM^*)^{*'}$ and denote its rank function by $\widehat{\rho}$.
Note that $\widehat{\mcM} \sim \mcM$ since they are both dual $q$-matroids of $\mcM^*$ w.r.t. different inner products.
Hence, according to Proposition \ref{prop:conequiv},  $\widehat{\mcM}$ has a vertical $t$-separation if and only if $\mcM$ has a vertical $t$-separation. We therefore show that, as constructed, $(A,V)$ is a $t$-separation of $\widehat{\mcM}$, where $t \leq \widehat{\rho}(E) -1$
Since $B \leq A$ is a circuit of $\mcM^*$, $A$ must be dependent and thus  $\rho^*(A) \leq \dim(A) - 1$. Therefore, 
\begin{align*}
    \widehat{\rho}(V) &= \dim V + \rho^*(V^{\perp'}) - \rho^*(E)\\
                  &= \dim V + \rho^*(A) - n + \widehat{\rho}(E)\\
                  &= \dim V + \dim A - 1 -n + \widehat{\rho}(E)\\
                  &= \widehat{\rho}(E) -1.
\end{align*}
In a similar way, $\widehat{\rho}(A) \leq \widehat{\rho}(E) -1$.
Let $t:= \min\{\widehat{\rho}(A), \widehat{\rho}(V)\} \leq \rho(E) -1.$  Without loss of generality, let $t = \widehat{\rho}(A)$. Clearly, $t \leq \min\{\widehat{\rho}(A), \widehat{\rho}(V)\}$. Furthermore we get
\begin{align*}
\widehat{\rho}(A) + \widehat{\rho}(V) - \widehat{\rho}(E) &\leq \widehat{\rho}(A) + \widehat{\rho}(E) -1 - \widehat{\rho}(E)\\
&= \widehat{\rho}(A) -1 = t-1.
\end{align*}
This shows that $(A, V)$ is a $t$-separation of $\widehat{\mcM}$.
\end{proof}

\begin{remark}
    Note in Theorem \ref{prop:conequiv}, one can only conclude the existence of two disjoint circuits in $\mcM^*$  given a vertical $t$-separation in $\mcM$ and vice versa. Unlike Proposition \ref{prop:disjcirc} for classical matroids, it is not true that given a $t$-separation $(A,V)$ of a $q$-matroid $\mcM$, that the disjoint circuits of $\mcM^*$ will respectively contained in $A$ and $V$ regardless of the inner product used to define $\mcM^*$. From the proof of Theorem \ref{prop:conequiv}, we can only conclude that there exists at least one inner product such that both $A$ and $V$  respectively contain circuits of $\mcM^*$ with respect to that inner product.   
\end{remark}

If we reformulate the previous results in terms of vector rank-metric codes, we obtain a $q$-analog of Theorem \ref{thm:conn_int}.

\begin{theorem}
    Let $\C \leq \F_{q^m}^n$ and $\mcM$ its induced $q$-matroid. Then $\mcM$  has a $t$-vertical separation if and only if $\C$ has two minimal codewords with disjoint support. Equivalently, $\C$ is intersecting if and only if $\kappa(\mcM)=\dim(\C)$.
\end{theorem}

\begin{proof}
 The code $\C \leq \F_{q^m}^n$ is not intersecting if and only if, by definition, there exists two minimal codewords with disjoint support. The latter, by \cite[Lemma 5.10]{alfarano2024cyclic}, is true if and only if $\mcM^*$ has two disjoint circuits which by Theorem \ref{prop:conequiv}, is true if and only if $\mcM$ has a $t$-separation with $t \leq \rho(E)-1 = k-1$.
\end{proof}

We conclude the paper with a result that links vertical connectivity of $q$-matroids to vertical connectivity of classical matroids.

\begin{theorem}\label{q_sep}
    Let $\mathcal M=(E, \rho)$ be a $q$-matroid with $\kappa(\mcM)=t$. Let $(A, V)$ be a $t$-separation for~$\mcM$. Fix a basis  $\beta = \{a_1, \ldots, a_t, v_{t+1}, \ldots, v_n\}$ for $E$, with $\{a_i\}$ a basis of $A$ and $\{v_i\}$ a basis for $V$. Define a (classical) matroid  $M=(\beta, r)$ with ground set $\beta$ and rank function given by 
    \[r(X)=\rho(\langle X \rangle).\]
    Then $\mcM$ and $M$ have the same vertical connectivity, each in their respective sense. 
    
\end{theorem}
\begin{proof}
We begin by showing that the pair $(\beta, r)$ is indeed a matroid in the classical sense. Let $X, Y \subseteq \beta$, then:
\begin{itemize}
        \item[(R1)]  $0 \leq r(X) = \rho(\langle X\rangle) \leq \dim(\langle X \rangle)=|X|$,
        \item[(R2)] if $X \subseteq Y$ then $r(X) =  \rho(\langle X\rangle) \leq  \rho(\langle Y\rangle) = r(Y)$,
        \item[(R3)] $ 
        \begin{aligned}[t]
            r(X \cap Y) + r(X \cup Y) & \leq r(A) + r(B) = \rho(\langle X \cap Y\rangle) + \rho(\langle X \cup Y \rangle) \\
            & \leq \rho(\langle X\rangle  \cap \langle Y\rangle)+\rho(\langle X\rangle  + \langle Y\rangle)  \\ & \leq \rho(\langle X\rangle) +\rho(\langle Y\rangle)=r(X) + r(Y),
        \end{aligned}
        $
    \end{itemize}
    where we have  used axioms ($q$R1), ($q$R2), ($q$R3), and the fact that $X\subseteq \beta$ is a linearly independent set.
    Now observe that:
    \begin{itemize}
            \item $r(\{a_1, \ldots, a_t\} = \rho(\langle a_1, \ldots, a_t\rangle) \geq t$,
            \item $r(\{v_{t+1}, \ldots, v_n\} )= \rho(\langle v_{t+1}, \ldots, v_n\rangle) \geq t$,
            \item $ \begin{aligned}[t] \lambda_{M}(\{a_1, \ldots a_t\}, \{v_{t+1}, \ldots v_n\}) &= \rho(\langle a_1, \ldots, a_t\rangle) + \rho(\langle v_{t+1}, \ldots, v_n \rangle) - \rho(\langle \beta \rangle) \\ &= \lambda_{\mathcal M}(A, V)<t, \end{aligned} $
    \end{itemize}
since $(A, V)$ is a separation for $\mathcal M$. Therefore $(\{a_1, \ldots a_t\}, \{v_{t+1}, \ldots v_n\})$ is a vertical $t$-separation for the classical matroid $M$. This shows that $\kappa(M)\leq\kappa(\mcM)$.  On the other hand, let $(X, X^c)$ be an $s$-separation for $M$. Then $(\langle X \rangle, \langle X^c \rangle)$ is an $s$-separation for $\mcM$, since 
\begin{itemize}
    \item $\rho(\langle X \rangle)=r(X)\geq s$, and likewise for $X^c$,
    \item $\lambda_\mcM(\langle X \rangle, \langle X^c \rangle) =  \rho(\langle X \rangle)+ \rho(\langle X^c \rangle) - \rho(E) = r(X)+r(X^c)-r(\beta)<s.$
\end{itemize}
Therefore $\kappa(\mcM)\leq \kappa(M)$, concluding the proof.
\end{proof}

\medskip

\bibliographystyle{unsrt}
\bibliography{biblio}

@article{alfarano2019geometric,
title = {A geometric characterization of minimal codes and their asymptotic performance},
journal = {Advances in Mathematics of Communications},
volume = {16},
number = {1},
pages = {},
year = {2022},
issn = {1930-5346},
doi = {10.3934/amc.2020104},
url = {https://www.aimsciences.org/article/id/d52e2e3a-722a-406f-b487-bf13dfc59caf},
author = {Gianira N. Alfarano and Martino Borello and Alessandro Neri}
}

@article{HENG2018176,
title = {Minimal linear codes over finite fields},
journal = {Finite Fields and Their Applications},
volume = {54},
pages = {},
year = {2018},
issn = {1071-5797},
doi = {https://doi.org/10.1016/j.ffa.2018.08.010},
url = {https://www.sciencedirect.com/science/article/pii/S1071579718301047},
author = {Ziling Heng and Cunsheng Ding and Zhengchun Zhou}
}

@article{KatonaSrivastava1983,
  title={Minimal 2-coverings of a finite affine space based on {GF}(2)},
  author={Katona, G. and Srivastava, J.},
  journal={Journal of Statistical Planning and Inference},
  volume={8},
  number={3},
  pages={},
  year={1983},
  publisher={Elsevier},
}

@article{blackburn,
title = {Frameproof Codes},
author = {S.R. Blackburn},
year = {2003},
doi = {10.1137/S0895480101384633},
volume = {16},
pages = {},
journal = {SIAM Journal on Discrete Mathematics},
issn = {1095-7146},
publisher = {Society for Industrial and Applied Mathematics Publications},
}

@article{alfarano2020combinatorialperspectivesminimalcodes,
author = {Alfarano, Gianira N. and Borello, Martino and Neri, Alessandro and Ravagnani, Alberto},
title = {Three Combinatorial Perspectives on Minimal Codes},
journal = {SIAM Journal on Discrete Mathematics},
volume = {36},
number = {1},
pages = {},
year = {2022},
doi = {10.1137/21M1391493},
URL = {https://doi.org/10.1137/21M1391493},
}

@article{COHNEN200375,
title = {Intersecting codes and separating codes},
journal = {Discrete Applied Mathematics},
volume = {128},
number = {1},
pages = {},
year = {2003},
issn = {0166-218X},
doi = {https://doi.org/10.1016/S0166-218X(02)00437-7},
url = {https://www.sciencedirect.com/science/article/pii/S0166218X02004377},
author = {G Cohen and S Encheva and S Litsyn and H.G Schaathun}
}

@article{MDSweights,
author = {Ezerman, Martianus and Grassl, Markus and Solé, Patrick},
year = {2009},
month = {},
pages = {},
title = {The weights in {MDS} codes},
volume = {57},
issue={1},
journal = {IEEE Transactions on Information Theory},
doi = {10.1109/TIT.2010.2090246}
}

@article{borello2024geometryintersectingcodesapplications,
title = {The geometry of intersecting codes and applications to additive combinatorics and factorization theory},
journal = {Journal of Combinatorial Theory, Series~A},
volume = {214},
pages = {},
year = {2025},
issn = {0097-3165},
doi = {https://doi.org/10.1016/j.jcta.2025.106023},
url = {https://www.sciencedirect.com/science/article/pii/S0097316525000184},
author = {Martino Borello and Wolfgang Schmid and Martin Scotti}
}

@article{ashbarg,
  author={Ashikhmin, A. and Barg, A.},
  journal={IEEE Transactions on Information Theory}, 
  title={Minimal vectors in linear codes}, 
  year={1998},
  volume={44},
  number={5},
  pages={},
  doi={10.1109/18.705584}}

@article{RavagnaniByrne2018,
title = {Partition-balanced families of codes and asymptotic enumeration in coding theory},
journal = {Journal of Combinatorial Theory, Series~A},
volume = {171},
pages = {},
year = {2020},
issn = {0097-3165},
doi = {https://doi.org/10.1016/j.jcta.2019.105169},
url = {https://www.sciencedirect.com/science/article/pii/S0097316519301505},
author = {Eimear Byrne and Alberto Ravagnani},

}

@book{Oxley,
    author = {Oxley, James},
    title = {Matroid Theory},
    publisher = {Oxford University Press},
    year = {2011},
    month = {},
    isbn = {9780198566946},
    doi = {10.1093/acprof:oso/9780198566946.001.0001},
    url = {https://doi.org/10.1093/acprof:oso/9780198566946.001.0001},
}

@article{hiltonmilner,
    author = {Hilton, A. J. W. and Milner, E. C.},
    title = {Some intersection theorems for systems of finite sets},
    journal = {The Quarterly Journal of Mathematics},
    volume = {18},
    number = {1},
    pages = {},
    year = {1967},
    month = {},
    issn = {0033-5606},
    doi = {10.1093/qmath/18.1.369},
    url = {https://doi.org/10.1093/qmath/18.1.369},
    eprint = {https://academic.oup.com/qjmath/article-pdf/18/1/369/4498834/18-1-369.pdf},
}

@article{erdoskorado,
    author = {Erd{\H{o}}s, P. and Ko, Chao and Rado, R.},
    title = {Intersection Theorems for Systems of Finite Sets},
    journal = {The Quarterly Journal of Mathematics},
    volume = {12},
    number = {1},
    pages = {},
    year = {1961},
    month = {},
    issn = {0033-5606},
    doi = {10.1093/qmath/12.1.313},
    url = {https://doi.org/10.1093/qmath/12.1.313},
    eprint = {https://academic.oup.com/qjmath/article-pdf/12/1/313/7288156/12-1-313.pdf},
}

@article{bartoli2025linearrankmetricintersectingcodes,
      title={Linear rank-metric intersecting codes}, 
      author={Daniele Bartoli and Martino Borello and Giuseppe Marino and Martin Scotti},
      year={2025},
      eprint={2507.00569},
      archivePrefix={arXiv},
      primaryClass={math.CO},
      journal={Preprint},
      note = {\href{https://arxiv.org/abs/2507.00569}{arXiv:2507.00569}}
}

@Incollection{Gorla2018,
author="Gorla, Elisa
and Ravagnani, Alberto",
title="Codes Endowed with the Rank Metric",
bookTitle="Network Coding and Subspace Designs",
year="2018",
publisher="Springer",
isbn="978-3-319-70293-3",
doi="10.1007/978-3-319-70293-3_1",
url="https://doi.org/10.1007/978-3-319-70293-3_1"
}

@article{Jurriusqmatroid,
author = {Jurrius, Relinde and Pellikaan, R.},
year = {2018},
month = {},
journal = {The Electronic Journal of Combinatorics},
volume = {25},
issue={3},
pages = {},
title = {Defining the $q$-analogue of a matroid},
doi = {10.48550/arXiv.1610.09250}
}

@article{degen2024most,
  title={Most $q$-matroids are not representable},
  author={Degen, Sebastian and K{\"u}hne, Lukas},
  journal={Preprint},
  note = {\href{https://arxiv.org/abs/2408.06795}{arXiv:2408.06795}},
  year={2024}
}

@article{jany2023projectivization,
  title={The Projectivization Matroid of a-Matroid},
  author={Jany, Benjamin},
  journal={SIAM Journal on Applied Algebra and Geometry},
  volume={7},
  number={2},
  pages={},
  year={2023},
  publisher={SIAM}
}

@article{gluesing2025cloud,
  title={The Cloud and Flock Polynomials of $q$-Matroids},
  author={Gluesing-Luerssen, Heide and Jany, Benjamin},
  journal={Preprint},
  note = {\href{https://arxiv.org/abs/2501.14984}{arXiv:2501.14984}},
  year={2025}
}

@article{alfarano2023critical,
  title={The critical theorem for $q$-polymatroids},
  author={Alfarano, Gianira N. and Byrne, Eimear},
  journal={Preprint},
  note = {\href{https://arxiv.org/abs/2305.07567}{arXiv:2305.07567}},
  year={2023}
}

@article{ghorpade2022shellability,
  title={Shellability and homology of $q$-complexes and $q$-matroids},
  author={Ghorpade, Sudhir R. and Pratihar, Rakhi and Randrianarisoa, Tovohery H.},
  journal={Journal of Algebraic Combinatorics},
  volume={56},
  number={4},
  pages={},
  year={2022},
  publisher={Springer}
}

@article{byrne2023cyclic,
  title={The cyclic flats of $\mathcal{L}$-polymatroids},
  author={Byrne, Eimear and Fulcher, Andrew},
  journal={Preprint},
  note = {\href{https://arxiv.org/abs/2312.05522}{arXiv:2312.05522}},
  year={2023}
}

@article{alfarano2024cyclic,
  title={The cyclic flats of a $q$-matroid},
  author={Alfarano, Gianira N. and Byrne, Eimear},
  journal={Journal of Algebraic Combinatorics},
  volume={60},
  number={1},
  pages={},
  year={2024},
  publisher={Springer}
}

@article{byrne2022constructions,
  title={Constructions of new $q$-cryptomorphisms},
  author={Byrne, Eimear and Ceria, Michela and Jurrius, Relinde},
  journal={Journal of Combinatorial Theory, Series~B},
  volume={153},
  pages={},
  year={2022},
  publisher={Elsevier}
}

@article{cunningham1981matroid,
  title={On matroid connectivity},
  author={Cunningham, William H},
  journal={Journal of Combinatorial Theory, Series~B},
  volume={30},
  number={1},
  pages={},
  year={1981},
  publisher={Elsevier}
}

@article{Oxley1981,
    author = {Oxley, James G.},
    title = {ON MATROID CONNECTIVITY},
    journal = {The Quarterly Journal of Mathematics},
    volume = {32},
    number = {2},
    pages = {},
    year = {1981},
    month = {},
    issn = {0033-5606},
    doi = {10.1093/qmath/32.2.193},
    url = {https://doi.org/10.1093/qmath/32.2.193},
    eprint = {https://academic.oup.com/qjmath/article-pdf/32/2/193/4434604/32-2-193.pdf},
}

@article{Inukaivert,
author = {Inukai, Thomas and Weinberg, Louis},
title = {Whitney Connectivity of Matroids},
journal = {SIAM Journal on Algebraic Discrete Methods},
volume = {2},
number = {2},
pages = {},
year = {1981},
doi = {10.1137/0602014},
URL = { 
        https://doi.org/10.1137/0602014
},
eprint = { 
        https://doi.org/10.1137/0602014
}
}

@article {bartoli2019inductive,
    AUTHOR = {Bartoli, Daniele and Bonini, Matteo and G\"{u}ne\c{s}, Bur\c{c}in},
     TITLE = {An inductive construction of minimal codes},
   JOURNAL = {Cryptography and Communications},
  FJOURNAL = {Cryptography and Communications. Discrete Structures, Boolean
              Functions and Sequences},
    VOLUME = {13},
      YEAR = {2021},
    NUMBER = {3},
     PAGES = {},
      ISSN = {1936-2447},
   MRCLASS = {94B05 (94A62)},
  MRNUMBER = {4258045},
       DOI = {10.1007/s12095-021-00474-2},
       URL = {https://doi.org/10.1007/s12095-021-00474-2},
}

@article{bartoli2020cutting,
  url = {https://doi.org/10.1515/forum-2020-0338},
title = {On cutting blocking sets and their codes},
author = {Daniele Bartoli and Antonio Cossidente and Giuseppe Marino and Francesco Pavese},
pages = {},
volume = {34},
number = {2},
journal = {Forum Mathematicum},
doi = {doi:10.1515/forum-2020-0338},
year = {2022},
lastchecked = {2026-02-13}
}

@article {bonini2020minimal,
    AUTHOR = {Bonini, Matteo and Borello, Martino},
     TITLE = {Minimal linear codes arising from blocking sets},
   JOURNAL = {Journal of Algebraic Combinatorics},
  FJOURNAL = {Journal of Algebraic Combinatorics. An International Journal},
    VOLUME = {53},
      YEAR = {2021},
    NUMBER = {2},
     PAGES = {},
      ISSN = {0925-9899},
   MRCLASS = {94B05 (51E21 94A62 94C11)},
  MRNUMBER = {4238182},
       DOI = {10.1007/s10801-019-00930-6},
       URL = {https://doi.org/10.1007/s10801-019-00930-6},
}

@incollection {MR3163591,
    AUTHOR = {Cohen, G\'{e}rard D. and Mesnager, Sihem and Patey, Alain},
     TITLE = {On minimal and quasi-minimal linear codes},
 BOOKTITLE = {Cryptography and coding},
    SERIES = {Lecture Notes in Computer Science},
    VOLUME = {8308},
     PAGES = {},
 PUBLISHER = {Springer, Heidelberg},
      YEAR = {2013},
   MRCLASS = {94B05},
  MRNUMBER = {3163591},
MRREVIEWER = {Simon N. Litsyn},
       DOI = {10.1007/978-3-642-45239-0_6},
       URL = {https://doi.org/10.1007/978-3-642-45239-0_6},
}

@article{kiermaier2019minimum,
  title={On the minimum number of minimal codewords},
  author={dela Cruz, Romar and Kiermaier, Michael and Kurz, Sascha and Wassermann, Alfred},
  journal={Advances in Mathematics of Communications},
  year={2023},
  volume={17},
  issue={2}
}

@article {lu2019parameters,
    AUTHOR = {Lu, Wei and Wu, Xia and Cao, Xiwang},
     TITLE = {The parameters of minimal linear codes},
   JOURNAL = {Finite Fields and Their Applications},
  FJOURNAL = {Finite Fields and their Applications},
    VOLUME = {71},
      YEAR = {2021},
     PAGES = {},
      ISSN = {1071-5797},
   MRCLASS = {94B05 (94A62)},
  MRNUMBER = {4202557},
       DOI = {10.1016/j.ffa.2020.101799},
       URL = {https://doi.org/10.1016/j.ffa.2020.101799},
}

@article {Massey,
    AUTHOR = {Massey, J. L. },
     TITLE = {Minimal codewords and secret sharing},
   JOURNAL = {In Proceedings of the 6th joint Swedish-Russian international workshop on information theory},
    VOLUME = {},
      YEAR = {1993},
    NUMBER = {},
     PAGES = {},
}

@article {tang2019full,
    AUTHOR = {Tang, C. and Qiu, Y. and Liao, Q. and Zhou, Z.},
     TITLE = {Full characterization of minimal linear codes as cutting blocking sets},
   JOURNAL = {IEEE Transactions on Information Theory},
  FJOURNAL = {Institute of Electrical and Electronics Engineers.
              Transactions on Information Theory},
    VOLUME = {67},
      YEAR = {2021},
    NUMBER = {6},
     PAGES = {},
       DOI = {10.1109/TIT.2021.3070377},
}

@article{barg,
  title={The matroid of supports of a linear code},
  author={Barg, A.},
  journal={Applicable Algebra in Engineering, Communication and Computing},
  volume={8},
  number={2},
  pages={},
  year={1997}
}

@article{britz2007higher,
  title={Higher support matroids},
  author={Britz, T.},
  journal={Discrete Mathematics},
  volume={307},
  number={17-18},
  pages={},
  year={2007}
}

@incollection{jurrius2013codes,
  title={Codes, arrangements and matroids},
  author={Jurrius, R. and Pellikaan, R.},
  booktitle={Algebraic {G}eometry {M}odeling in {I}nformation {T}heory},
  pages={},
  year={2013},
  publisher={World Scientific}
}

@article{britz2010code,
  title={Code enumerators and {T}utte polynomials},
  author={Britz, T.},
  journal={IEEE Transactions on Information Theory},
  volume={56},
  number={9},
  pages={},
  year={2010}
}

@article{J2013,
  title={Hamming weights and {B}etti numbers of {S}tanley-{R}eisner rings associated to matroids},
  author={Johnsen, T. and Verdure, H.},
  journal={Applicable Algebra in Engineering, Communication and Computing},
  volume={24},
  number={1},
  pages={},
  year={2013}
}

\end{document}